\documentclass{article}
\usepackage{amsopn} % provides \DeclareOperation
\usepackage{amsfonts} % provides \mathbb
\usepackage{amsmath} % provides \lVert.  Comes with amsgen, amsbsy, amsopn, amstext
\usepackage{amssymb} %needed for \supsetneqq
\usepackage{amsthm} % provides the proof environment and \newtheorem
\usepackage{amscd}
\usepackage{enumerate} % provides the enumerate environment needed for enumeratei
\usepackage{verbatim}
\usepackage{url}
\usepackage{xspace}
\usepackage{pgf}
\usepackage{tikz}
\usetikzlibrary{arrows,automata}
\usepackage{subfig}
\newtheorem{lemma}{Lemma}
\newtheorem{example}{Example}
\newtheorem{definition}{Definition}
\newtheorem{remark}{Remark}
\newtheorem{theorem}{Theorem}

\renewcommand{\epsilon}{\varepsilon}

\newcommand{\calC}{\ensuremath{\mathcal C}\xspace}

\newcommand{\calL}{\ensuremath{\mathcal{L}}\xspace}

\newcommand{\calN}{\ensuremath{\mathcal{N}}\xspace}

\newcommand{\calS}{\ensuremath{\mathcal{S}}\xspace}

\newcommand{\calX}{\ensuremath{\mathcal{X}}\xspace}
\newcommand{\calY}{\ensuremath{\mathcal{Y}}\xspace}

\newcommand{\N}{\ensuremath{\mathbb{N}}\xspace}

\newcommand{\bff}{\ensuremath{\mathbf{f}}\xspace}

\newcommand{\bfs}{\ensuremath{\mathbf{s}}\xspace}
\newcommand{\bft}{\ensuremath{\mathbf{t}}\xspace}

\newcommand{\bfx}{\ensuremath{\mathbf{x}}\xspace}
\newcommand{\bfy}{\ensuremath{\mathbf{y}}\xspace}
\newcommand{\bfz}{\ensuremath{\mathbf{z}}\xspace}

\newcommand{\nsset}[1]{\{#1\}} % normal singular set

\newcommand{\bset}[2]{\bigl\{\,#1\mid #2 \,\bigr\}}
 
\newcommand{\nparen}[1]{(#1)}
\newcommand{\bparen}[1]{\bigl(#1\bigr)}

\newcommand{\nabs}[1]{\lvert #1\rvert}

\newcommand{\qtext}[1]{\quad\text{#1}\quad}
\newcommand{\qqtext}[1]{\qquad\text{#1}\qquad}

\newcommand{\ntoinf}{n \rightarrow \infty}

\renewcommand{\o}{\ensuremath{\mathtt{0}}\xspace}
\newcommand{\oi}{\ensuremath{\mathtt{01}}\xspace}
\newcommand{\io}{\ensuremath{\mathtt{10}}\xspace}
\renewcommand{\i}{\ensuremath{\mathtt{1}}\xspace}
\newcommand{\ka}{\ensuremath{\mathtt{2}}\xspace}
\newcommand{\ko}{\ensuremath{\mathtt{3}}\xspace}

\newcommand{\bina}{\nsset{\o,\i}}
\newcommand{\terna}{\nsset{\o,\i, \ka}}
\newcommand{\kvarta}{\nsset{\o,\i, \ka, \ko}}

\newcommand{\talf}{\reflectbox{${}^{\flat}$}}
\begin{document}
\title{Suffix conjugates for a class of morphic subshifts}
\author{James D. Currie\thanks{j.currie@uwinnipeg.ca}
\and Narad Rampersad \thanks{narad.rampersad@gmail.com}
\and Kalle Saari \thanks{kasaar2@gmail.com}
\\
Department of Mathematics and Statistics\\University of Winnipeg\\515 Portage Avenue\\Winnipeg, MB R3B 2E9\\Canada
}

\maketitle

\begin{abstract}
Let $A$ be a finite alphabet and $f \colon A^{*} \rightarrow A^{*}$ be a morphism with an iterative fixed point $f^{\omega}(\alpha)$,
where $\alpha \in A$. Consider the subshift $(\calX, T)$, where $\calX$ is the shift orbit closure of $f^{\omega}(\alpha)$ and $T\colon \calX \rightarrow \calX$ is the shift map.
Let $S$ be a finite alphabet that is in bijective correspondence via a mapping $c$ with the set of nonempty suffixes of the images $f(a)$ for $a\in A$.
Let $\calS\subset S^{\N}$ be the set of infinite words $\bfs = (s_{n})_{n\geq0}$ such that 
$\pi\nparen{\bfs} := c(s_{0})f\bparen{c(s_{1})}f^{2}\bparen{c(s_{2})}\cdots \in \calX$. 
We show that if $f$ is primitive and $f(A)$ is a suffix code, then 
there exists a mapping $H\colon \calS \rightarrow \calS$ such that $(\calS, H)$ is a topological dynamical system 
and $\pi \colon (\calS, H) \rightarrow (\calX, T)$ is a conjugacy; we call $(\calS, H)$ the \emph{suffix conjugate} of $(\calX, T)$.
In the special case when $f$ is the Fibonacci or the Thue-Morse morphism, we show that
the subshift $(\calS, T)$ is sofic, that is, the language of $\calS$ is regular.
\end{abstract}

\section{Introduction}

Let $A$ be a finite alphabet and $f \colon A^{*} \rightarrow A^{*}$ a morphism with an iterative fixed point $f^{\omega}(\alpha) = \lim_{\ntoinf}f^{n}(\alpha)$. Consider the shift orbit closure $\calX$ generated by~$f^{\omega}(\alpha)$. 
If $\bfx \in\calX$, then there exist a letter $a \in A$ and an infinite word $\bfy \in \calX$ such that
$\bfx = s f(\bfy)$, where $s$ is a nonempty suffix of $f(a)$~\cite[Lemma~6]{CurRamSaa2013}.
This formula has been observed several times in different contexts, see~\cite{HolZam2001} and the references therein.
Since $\bfy \in \calX$, this process can be iterated to generate an expansion
\begin{equation}\label{2013-04-10 21:15}
\bfx = s_{0}f(s_{1})f^{2}(s_{2}) \cdots f^{n}(s_{n}) \cdots,
\end{equation}
where each $s_{n}$ is a nonempty suffix of an image of some letter in~$A$. In general, however, not every sequence $(s_{n})_{n\geq0}$ of suffixes gives rise to an infinite word in $\calX$ by means of this kind of expansion.
Therefore, in this paper we introduce the set $\calS$ that consists of those $(s_{n})_{n\geq0}$ whose expansion~\eqref{2013-04-10 21:15} is in $\calX$.  Our goal is then to understand the structure of $\calS$. By endowing $\calS$ with the usual metric on infinite words, $\calS$ becomes a metric space. Furthermore, $\calS$ can be associated with 
a mapping $G \colon \calS \rightarrow \calS$ (see below) giving rise to a topological dynamical system 
$(\calS, G)$ that is an extension of $(\calX, f)$; see the discussion around Eq.~\eqref{2013-04-12 19:15}
However, imposing some further restrictions on $f$, we obtain a much stronger result:
If $f$ is a circular morphism such that $\nabs{f^{n}(a)}\rightarrow \infty$ for all $a \in A$ and $f(A)$ is a suffix code,
then there exists a mapping $H\colon \calS \rightarrow \calS$ such that $(\calS, H)$ and $(\calX, T)$, where $T$ is the usual shift operation, are conjugates (Theorem~\ref{2013-04-12 22:44}). We call $(\calS, H)$ the \emph{suffix conjugate} of $(\calX, T)$. Since primitive morphisms are circular (i.e., recognizable) by Moss\'e's theorem~\cite{Mosse1992}, primitivity of $f$ together with the suffix code condition suffice for the existence of the
suffix conjugate. In particular, both the Fibonacci morphism $\varphi\colon \o \mapsto \o\i$, $\i \mapsto \o$
and the Thue-Morse morphism $\mu\colon \o \mapsto\o\i$, $\i \mapsto \i\o$ satisfiy these conditions, 
and so the corresponding Fibonacci subshift $(\calX_{\varphi}, T)$ and the Thue-Morse subshift $(\calX_{\mu}, T)$ have suffix conjugates.
In this paper we characterize the language of both subshifts and show that they are regular. 

An encoding scheme for $\calX$ related to ours was considered by Holton and Zamboni~\cite{HolZam2001}
and Canterini and Siegel~\cite{CanSie2001}, who studied 
bi-infinite primitive morphic subshifts and essentially used prefixes of images of letters where we use suffixes. Despite of this seemingly insignificant difference, though,
we are not aware of any mechanism that would allow transferring results from one encoding scheme to another.
See also the work by Shallit~\cite{Shallit2011}, who constructed a finite automaton that provides an encoding for the set of infinite overlap-free words.

\section{Preliminaries and generalities}

In this paper we will follow the standard notation and terminology of combinatorics on 
words~\cite{Lothaire2002,AllSha2003} and symbolic dynamics~\cite{LinMar1995,Kurka2003}.

Let $A$ be a finite alphabet and $f : A^{*} \rightarrow A^{*}$ a morphism
with an iterative fixed point $f^{\omega}(\alpha) = \lim_{\ntoinf}f^{n}(\alpha)$, where $\alpha \in A$.
Let $\calX$ be the shift orbit closure generated by~$f^{\omega}(\alpha)$.
Let $S'$ be the set of nonempty suffixes of images of letters under~$f$.
Denote $S = \nsset{0, 1, \ldots, \nabs{S'} - 1}$
and let $c\colon S \rightarrow S'$ be a bijection. We consider $S$ as a finite alphabet.

If $s = s_{0}s_{1}\cdots s_{n}$ with $s_{i} \in S$, then we denote by $\pi(s)$ the word 
\[
\pi(s) = c(s_{0})f(c(s_{1})) f^{2}(c(s_{2})) \cdots f^{n}(c(s_{n})) \in A^{*}.
\]
Then $\pi$ extends to a mapping $\pi \colon S^{\N} \rightarrow A^{\N}$ in a natural way,
and so we may define
\[
\calS = \bset{ \bfs \in S^{\N}}{ \pi(\bfs) \in \calX }.
\]
Our goal in this section is to find sufficient conditions on $f$ so that $\calS$ can be endowed with dynamics
that yields a conjugate to $(\calX, T)$ via the mapping~$\pi$. Examples~\ref{2013-04-12 01:04} and~\ref{2013-04-12 09:57} below show that this task is not trivial. Such sufficient conditions are laid out in Definition~\ref{2013-04-10 17:36}.

If $\bfx \in \calX$ and $\bfs \in \calS$ such that $\pi(\bfs) = \bfx$, we say that 
$\bfx$ is an \emph{expansion} of $\bfs$.

\begin{lemma}[Currie, Rampersad, and Saari \cite{CurRamSaa2013}]\label{2013-04-10 20:38}
For every $\bfx \in \calX$, there exist $a\in A$, a non-empty suffix $s$ of $f(a)$, and an infinite word $\bfy \in \calX$
such that $\bfx = s f(\bfy)$ and $a\bfy \in \calX$.
Therefore the mapping $\pi \colon \calS \rightarrow \calX$ is surjective.
\end{lemma}

Both $A^{\N}$ and $S^{\N}$ are endowed with the usual metric 
\[
d\bparen{(x_{n})_{n\geq0} , (y_{n})_{n\geq0}} = \frac{1}{2^{n}}, \qtext{where}  n = \inf\bset{n}{x_{n} \neq y_{n}},
\]
The following lemma is obvious.
\begin{lemma}\label{2013-04-10 20:45}
The mapping $\pi \colon \calS \rightarrow \calX$ is continuous.
\end{lemma}

We denote the usual shift operation $(x_{n})_{n\geq0} \mapsto (x_{n+1})_{n\geq0}$ in both spaces $A^{\N}$ and $S^{\N}$ by $T$.
We have $T(\calX) \subset \calX$ and $f(\calX) \subset \calX$  by the construction of $\calX$,
and both $T$ and $f$ are clearly continuous on $\calX$, so we have the topological dynamical systems $(\calX, T)$ and $(\calX, f)$. Note, however, that in general $T(\calS)$ is not necessarily a subset of $\calS$,
as the following example shows.
\begin{example}\label{2013-04-12 01:04}
Let $f \colon \nsset{\alpha, a, b}^{*} \rightarrow  \nsset{\alpha, a, b}^{*}$ be the morphism
$\alpha \mapsto \alpha ab$, $a \mapsto a$, and $b \mapsto ab$. Then 
\[
f^{\omega}(\alpha) = \alpha f(b) f^{2}(b) f^{3}(b) \cdots \qtext{and} 
T f^{\omega}(\alpha) = babaabaaab\cdots.
\]
Since the latter sequence is not in the shift orbit closure $\calX$ generated by $f^{\omega}(\alpha)$,
this shows that $\calS$ is not closed under $T$ for this particular morphism.
\end{example}

If $f$ is the morphism $\o \mapsto \o\i$, $\i \mapsto\o$, then $f$ is called the \emph{Fibonacci morphism}
and we write $f = \varphi$. The unique fixed point of $\varphi$ is denoted by $\bff$ and it is called the \emph{Fibonacci word}. 
The shift orbit closure it generates is denoted by $\calX_{\varphi}$ and the pair $(\calX_{\varphi}, T)$ is called the \emph{Fibonacci subshift}.

Similarly, if $f$ is $\o \mapsto \o\i$, $\i \mapsto\i\o$, then $f$ is the \emph{Thue-Morse morphism}
and we write $f = \mu$. The fixed point $\mu^{\omega}(\o)$ of $\mu$ is denoted by $\bft$ and it is called the \emph{Thue-Morse word}. The shift orbit closure  generated by $\bft$ is denoted by $\calX_{\mu}$, and the pair 
$(\calX_{\mu}, T)$ is called the \emph{Thue-Morse subshift}.

\begin{example}\label{2013-04-12 09:57}
Let $f$ be the morphism $0 \mapsto 010$, $1 \mapsto10$. 
The two fixed points of $f$ generate the Fibonacci subshift.
The set of suffixes of $f(\o)$ and $f(\i)$ is $S' = \nsset{\o, \i\o, \o\i\o}$,
and we define a bijection $c \colon \nsset{\o, \i, \ka} \rightarrow S'$
by $c(\o) = \o$, $c(\i) = \i\o$, and $c(\ka) = \o\i\o$. 
Then $\pi(\o\i) = \pi(\ka \o) = \o\i\o\o\i\o$, and therefore $\pi(\o \i^{\omega}) = \pi(\ka\o\i^{\omega})$.
This word  equals the Fibonacci word $\bff$ as can be seen by observing that 
\[
\bff = \o \i\o \varphi^{2}(\i\o) \varphi^{4}(\i\o) \varphi^{6}(\i\o)  \cdots
\]
and $010 f^{n}(a) = \varphi^{2n}(a)010$ for all $n\geq 0$ and $a\in \bina$.
This shows that it is possible for two distinct words in $\calS$ to have the same expansions,
and therefore $\pi$ is not always injective.
\end{example}

The following lemma is a straightforward consequence of the definition of $\pi$.

\begin{lemma}\label{2013-04-10 17:21}
Let $\bfs = s_{0}s_{1}s_{2}\cdots$, where $s_{i}\in S$. Then
\[
f\bparen{\pi\circ T(\bfs)} = T^{\nabs{c(s_{0})}} \pi(\bfs). 
\]
and
\begin{equation}\label{2013-04-02 14:43}
\pi(\bfs) = \pi(s_{0}s_{1}\cdots s_{n-1}) f^{n}\bparen{\pi(T^{n}\bfs)}.
\end{equation}
For finite words $x,y \in S^{*}$, the above reads $\pi(xy) = \pi(x)f^{\nabs{x}}\bparen{\pi(y)}$.
\end{lemma}

Note that if $s \in S$ such that $c(s) \in S'$ is a letter, then $f\bparen{c(s)} \in S'$.
As this connection will be frequently referred to, we define a morphism 
\begin{equation}\label{2013-04-13 20:46}
\lambda \colon S_{1}^{*}\rightarrow S^{*} \qqtext{with}
\lambda(s) = c^{-1}\bparen{f\bparen{c(s)}},
\end{equation}
where $S_{1} \subset S$ consists of those $s\in S$ for which $\nabs{c(s)} = 1$.
Then in particular, $c \bparen{\lambda(s)} = f\bparen{c(s)}$.

%In other words, we write $c^{-1}\bparen{f(s)}$ by $\lambda(s)$.

\begin{lemma} \label{2013-03-20T14:46}
Let $\bfs = s_{0}s_{1}\cdots \in \calS$ with $s_{i} \in S$, and write
$\bfx = \pi(\bfs) \in \calX$.
Let $r\geq0$ be the smallest integer, if it exists, such that $\nabs{c(s_{r})} \geq 2$ and write
$c(s_{r}) = au$, where $a\in A$ and $u\in A^{+}$.
Then $f(\bfx) = \pi(\bft)$, where $\bft  = t_{0}t_{1}\cdots \in \calS$  satisfies
\begin{itemize}
\item $t_{i} = \lambda(s_{i})$ for $i = 0, 1, \ldots, r - 1$,
\item $t_{r} = c^{-1}\bparen{f(a)}$,
\item $t_{r+1} = c^{-1}(u)$, and
\item $t_{i} = s_{i-1}$ for $i\geq r + 2$.
\end{itemize}
If each of $c(s_{i})$ is a letter, then $f(\bfx) = \pi(\bft)$, where 
\[
\bft = \lambda(s_{0}) \lambda(s_{1}) \cdots  \lambda(s_{n})  \cdots.
\]
\end{lemma}
\begin{proof}
Suppose $r$ exists.
The identity $\bfx = \pi(\bfs)$ says that
\[
\bfx = c(s_{0})f\bparen{c(s_{1})} \cdots f^{r-1}\bparen{c(s_{r-1})} f^{r}\bparen{c(s_{r})} f^{r+1}\bparen{c(s_{r+1})}\cdots
\]
Therefore, by denoting $f\bparen{c(s_{i})} = \hat{s}_{i} \in S'$ for $i=0,1,\ldots, r - 1$, we see that
\begin{align*}
f(\bfx) &= f\bparen{c(s_{0})}f^{2}\bparen{c(s_{1})} \cdots f^{r}\bparen{c(s_{r-1})} f^{r+1}\bparen{c(s_{r})} 
f^{r+2}\bparen{c(s_{r+1})} \cdots \\
		&= \hat{s}_{0}f(\hat{s}_{1}) \cdots f^{r-1}(\hat{s}_{r-1}) f^{r+1}(au) f^{r+2}\bparen{c(s_{r+1})} \cdots   \\
		&= \hat{s}_{0}f(\hat{s}_{1}) \cdots f^{r-1}(\hat{s}_{r-1}) f^{r}\bparen{f(a)} f^{r+1}(u) f^{r+2}\bparen{c(s_{r+1})} \cdots\\
		&= c(t_{0}) f\bparen{c(t_{1})} f^{2}\bparen{c(t_{2})}\cdots,
\end{align*}
where the $t_{i}$'s are  as in the statement of the lemma.
%\begin{itemize}
%\item $t_{i} = c^{-1}(\hat{s}_{i}) = \lambda(s_{i})$ for $i = 0, 1, \ldots, r - 1$,
%\item $t_{r} = c^{-1}\bparen{f(a)} = \lambda\bparen{c^{-1}(a)}$,
%\item $t_{r+1} = c^{-1}(u)$, and
%\item $t_{i} = s_{i-1}$ for $i\geq r + 2$,
%\end{itemize}
%as claimed. 
The case when $r$ does not exist is a special case of the above.
\end{proof}

Let $\bfs \in \calS$ and $\bft\in \calS$ be defined as in the previous lemma.
This defines a mapping $G \colon\calS \rightarrow \calS$ for which $G(\bfs) = \bft$,
which is obviously continuous.
Thus we have a topological dynamical system $(\calS, G)$.
Furthermore, by the definition of $G$, we have
\begin{equation}\label{2013-04-12 19:15}
f \circ \pi = \pi \circ G.
\end{equation}
Therefore $ \pi \colon (\calS, G) \rightarrow (\calX, f)$ is a factor map because $\pi$ is surjective by Lemma~\ref{2013-04-10 20:38} and continuous by Lemma~\ref{2013-04-10 20:45}. 
We can get a more concise definition for $G$ if we extend the domain of $\lambda$ defined 
in~\eqref{2013-04-13 20:46} to $S$ as follows. If $s\in S \setminus S_{1}$, then $f\bparen{c(s)} = au$
with $a\in A$ and $u\in A^{+}$, and we define
\begin{equation}\label{2013-04-14 16:13}
\lambda(s) = c^{-1}\bparen{f(a)}c^{-1}(u).
\end{equation}
Then we have, for all $\bfs \in \calS$,
\begin{equation}\label{2013-04-13 21:25}
G(\bfs) =
\begin{cases}
\lambda(ps)\bft & \text{if $\bfs = ps\bft$ with $p\in S_{1}^{*}$ and $s \in S\setminus S_{1}$} \\
\lambda(\bfs)  & \text{if $\bfs \in S_{1}^{\N}$.}
\end{cases}
\end{equation}

We got this far without imposing any restrictions on $f$, but now we have to introduce some further concepts.

If $\calY$ is the shift orbit closure of some infinite word $\bfx$, then the set of finite factors of $\bfx$ is called the 
\emph{language} of $\calY$ or $\bfx$ and denoted by $\calL(\calY)$ or by $\calL(\bfx)$.

If $x$ is a finite word and $y$ a finite or infinite word and $x$ is a factor of $y$, we will express this by writing $x \subset y$. This handy notation has been used before at least in~\cite{HolZam1999}.

A key property we would like our morphism $f$ to have is called \emph{circularity}, which has various formulations and is also called \emph{recognizability}. We use the formulation of Cassaigne~\cite{Cassaigne1994} and Klouda~\cite{Klouda2012};
see also~\cite{MigSee1994,Kurka2003}. The morphism~$f$ 
whose fixed point generates the shift orbit closure~$\calX$
is called \emph{circular on $\calL(\calX)$}  if $f$ is injective on $\calL(\calX)$ and
there exists a \emph{synchronization delay} $\ell\geq 1$ such that if $w \in \calL(\calX)$ and $\nabs{w} \geq \ell$, then
it has a \emph{synchronizing point} $(w_{1}, w_{2})$ satisfying the following two conditions: First, $w = w_{1} w_{2}$. Second, 
\[
\forall v_{1}, v_{2} \in A^{*} 
\left[v_{1} w v_{2} \in f\bparen{\calL(\calX)} \Longrightarrow 
v_{1}w_{1}\in f\bparen{\calL(\calX)} \qtext{and} w_{2}v_{2} \in f\bparen{\calL(\calX)}\right].
\]

A well-known result due to Moss\'e~\cite{Mosse1992} (see also~\cite{Kurka2003}) says that a primitive morphism
 with an aperiodic fixed point is circular (or \emph{recognizable}).

\begin{definition}\label{2013-04-10 17:36}
We write $f \in \calN$  to indicate that $f \colon A^{*} \rightarrow A^{*}$ with an iterative
fixed point $f^{\omega}(\alpha)$ 
has the following properties.
\begin{enumerate}[(i)]
\item $f$ is circular on the language of $f^{\omega}(\alpha)$;
\item the set $f(A)$ is a suffix code; i.e., no image of a letter is a suffix of another;
\item each letter $a \in A$ is growing; i.e., $\nabs{f^{n}(a)} \rightarrow \infty$ as $n\rightarrow\infty$.
\end{enumerate}
\end{definition}

In particular, if $f$ is primitive and $f^{\omega}(\alpha)$ aperiodic, then $f$ is circular by Moss\'e's theorem,
and if in addition $f(A)$ is a suffix code, then $f\in \calN$.
Therefore both the Fibonacci morphism $\varphi$ and the Thue-Morse morphism $\mu$ are in~$\calN$.

In Example~\ref{2013-04-12 01:04} we saw that, in general, $\calS$ is not necessarily closed under the shift map $T$ for a general morphism~$f$. 
The next lemma shows, however, that if $f \in \calN$, this problem does not arise.

\begin{lemma}\label{2013-04-11 17:37}
If $f \in \calN$, then $T(\calS) \subseteq \calS$. Thus $(\calS, T)$ is a subshift.
\end{lemma}
\begin{proof}
Let $\bfs = s_{0}s_{1}\cdots\in \calS$; then $\pi(\bfs) \in \calX$.
Equation~\eqref{2013-04-02 14:43} says that $\pi(\bfs) = c(s_{0})f\bparen{\pi(T\bfs)}$,
and so $f\bparen{\pi(T\bfs)} \in \calX$.
Suppose that $\pi(T\bfs) \notin \calX$.

Since $f\in \calN$, it is circular. Let $\ell \geq 1$ be a synchronization delay for~$f$.
Note that $f^{n-1}(s_{n})$ occurs both in $\pi(T\bfs)$  and in $f^{\omega}(\alpha)$ for every $n\geq 1$.
Since also $\nabs{f^{n-1}(s_n)} \rightarrow \infty$ as $n\rightarrow\infty$ because $f\in \calN$,  it follows that there are arbitrarily long words in $\calL(\calX)$
that occur in infinitely many positions in $\pi(T\bfs)$. Therefore there exists a word  $zy \subset \pi(T\bfs)$ such that $z$ is not in $\calL(\calX)$,
$y \in \calL(\calX)$, and $\nabs{y} \geq \ell$.

Next, consider the word $f(zy) \subset f\bparen{\pi(T\bfs)}$. Since 
$f(y) \in \calL(\calX)$ and $\nabs{f(y)} \geq \ell$, the word $f(y)$ has a synchronizing point $(w_{1}, w_{2})$.
In particular, since $y \in \calL(\calX)$, there exists $y_{1}, y_{2}$ for which
$y=y_{1}y_{2}$, $f(y_{1}) = w_{1}$, and $f(y_{2}) = w_{2}$.
On the other hand, $f(zy) \in \calL(\calX)$ implies that we can write $f^{\omega}(\alpha) = put\bfx$ such that 
$f(zy) \subset f(ut)$ and $f(y) \subset f(t)$. Thus there exists $t_{1}, t_{2}$ such that $t=t_{1}t_{2}$,
the word $w_{1}$ is a suffix of $f(t_{1})$, and $w_{2}$ is a prefix of $f(t_{2})$. Thus $f(y_{1})$ is a suffix of $f(t_{1})$.
Since $f(A)$ is a suffix code and $f$ is injective, it follows that $y_{1}$ is a suffix of $t_{1}$, and furthermore that $zy_{1}$ is a suffix of $ut_{1}$.
But then $z \in \calL(\calX)$ contradicting the choice of $z$.
Therefore $\pi(T\bfs) \in \calX$ and so $T\bfs \in \calS$.
\end{proof}

\begin{lemma}\label{2013-04-10 20:39}
If  $f \in \calN$, then the mapping $\pi \colon \calS \rightarrow \calX$ is injective.
\end{lemma}
\begin{proof}
For every $u,v \in A^{*}$ and $\bfx, \bfy \in \calX$, we have that $uf(\bfx) = vf(\bfy)$
implies $u = v$ and $\bfx = \bfy$. 
This follows from the circularity and suffix code property of~$f$. 
(See also the proof of Lemma~\ref{2013-04-11 17:37}.)
Therefore if $\bfs, \bfs' \in \calS$ and $\pi(\bfs) = \pi(\bfs')$, then Lemma~\ref{2013-04-10 17:21} gives  
\[
c(s_{0}) f\bparen{\pi(T\bfs)} = c(s'_{0}) f\bparen{\pi(T\bfs')},
\]
so that $c(s_{0}) = c(s'_{0})$ and $\pi(T\bfs) = \pi(T\bfs')$.
Thus $s_{0} = s'_{0}$, and since $T\bfs, T\bfs' \in \calS$ by Lemma~\ref{2013-04-11 17:37},
we can repeat the argument obtaining $s_{1} = s'_{1}$, $s_{2} = s'_{2}$, \ldots.
Therefore $\bfs = \bfs'$.
\end{proof}

\begin{remark}
In Lemma~\ref{2013-04-10 20:39} above, the assumption that $f$ is circular is crucial: If $f\colon a^{*}\rightarrow a^{*}$ is defined by $f(a) = aa$, then $\calS = S^{\N}$ while $\calX = \nsset{a^{\omega}}$,
so $\pi$ is anything but injective! Nevertheless, $f$ satisfies all conditions of $\calN$, except circularity.
\end{remark}

Now we are ready to define the desired dynamics on $\calS$.

\begin{theorem}\label{2013-04-12 22:44}
Suppose that $f \in \calN$. 
Let $H \colon \calS \rightarrow \calS$ be the mapping given by $H = T\circ G$.
Then $\pi \circ H = T \circ \pi$ and so $\pi \colon (\calS, H) \rightarrow (\calX, T)$ is a conjugacy.
\begin{center}
\begin{tikzpicture}[node distance = 2cm, auto]
\node(X0) {$\calS$}; 
\node(X1) [right of = X0] {$\calS$}; 
\node(E0) [below of = X0] {$\calX$}; 
\node(E1) [right of = E0] {$\calX$}; 

\draw[->] (X0) to node {$H$} (X1);
\draw[->] (E0) to node {$T$} (E1);
\draw[->] (X0) to node {$\pi$} (E0);
\draw[->] (X1) to node {$\pi$} (E1);
\end{tikzpicture}
\end{center} 

\end{theorem}
\begin{proof}
Observe first that  $H(\calS) \subset \calS$ by Lemma~\ref{2013-04-11 17:37},
so the definition of $H$ is sound. 
The mapping $\pi$ is surjective by Lemma~\ref{2013-04-10 20:38} and injective by Lemma~\ref{2013-04-10 20:39}, so it is a bijection. Furthermore $\pi$ is continuous by Lemma~\ref{2013-04-10 20:45}.
Finally, let us verify $\pi \circ H = T \circ \pi$. Let $\bfs = s_{0} s_{1} \cdots \in \calS$ with $s_{i} \in S$.
If $\nabs{c(s_{0})} \geq 2$, then we leave it to the reader to check that, by denoting $c(s_{0}) = au$ with $a\in A$, 
we have 
\[
\pi \circ H(\bfs) = \pi \circ T \circ G (\bfs) = u f\bparen{c(s_{1})} f^{2}\bparen{c(s_{2})} \cdots = T \circ \pi (\bfs). 
\]
If $\nabs{c(s_{0})} = 1$, then it is readily seen that $T \circ G (\bfs) = G \circ T(\bfs)$.
Using this, Equation~\eqref{2013-04-12 19:15}, and Lemma~\ref{2013-04-10 17:21} in this order gives
\[
\pi \circ H(\bfs) = \pi \circ T \circ G (\bfs)
				  = \pi \circ G \circ T (\bfs)
				  = f \circ \pi \circ T(\bfs)
				  = T^{\nabs{c_{0}}} \circ \pi(\bfs) = T \circ \pi(\bfs),
\]
and the proof is complete.
\end{proof}

The rest of this section is devoted to developing a few results for understanding the language of $\calS$.
They will be needed in the next sections that deal with the suffix conjugates of the Fibonacci and the Thue-Morse subshifts.

If $u$ is a finite nonempty word, we denote by~$u^{\flat}$ and $\talf u$ the words obtained from~$u$ by deleting its last and first letter, respectively. 

If a finite word $u$ is not in $\calL(\calX)$, then $u$ is called a \emph{forbidden word} of $\calX$. 
If both $\talf u$ and $u^{\flat}$ are in $\calL(\calX)$, then $u$ is a \emph{minimal forbidden word} of~$\calX$.
There is a connection between the minimal forbidden words and the so-called bispecial factors of an infinite word.
See a precise formulation of this in~\cite{MigResSci2002} and examples in Sections~\ref{2013-04-14 13:40} and~\ref{2013-04-14 13:38}.

We say that a word $u \in  S^{*}$ is a \emph{cover} of a word $v \in A^{*}$ if $v \subset \pi(u)$.
Furthermore, we say that the cover~$u$ is \emph{minimal} if 
$v \not\subset \pi(u^{\flat})$ and $v \not\subset f\bparen{\pi(\talf u)}$.
The latter expression comes from the identity $\pi(u) = c(u_{0}) f\bparen{\pi(\talf u)}$,
where $u_{0}$ is the first letter of~$u$, given by Lemma~\ref{2013-04-10 17:21}.

Let $\calC$ be the set of minimal covers of the minimal forbidden 
factors of~$\calX$.

\begin{lemma}\label{2013-04-11 12:48}
Suppose $f \in \calN$. Let $\bfs \in S^{\N}$. Then $\bfs \notin \calS$ if and only if $\bfs$ has a factor in $\calC$.
\end{lemma}
\begin{proof}
Suppose that $\bfs$ has a factor in $\calC$, so that $\bfs = pt\bfs'$ with $t \in \calC$. If $\bfs \in \calS$, then $T^{\nabs{p}} \bfs = t\bfs' \in \calS$ by Lemma~\ref{2013-04-11 17:37}. But $\pi(t\bfs')$ has prefix $\pi(t)$, in which 
a forbidden word occurs by the definition of $\calC$, a contradiction. 

Conversely, suppose that $\bfs \notin \calS$. Then $\pi(\bfs) \notin \calX$, so there exists a
minimal forbidden word $v_{0}$ of $\calX$ occurring in $\pi(\bfs)$. Let $u_{0}$ be the shortest prefix of $\bfs$ such that 
$v_{0} \subset \pi(u_{0})$. Then either $u_{0}$ is a minimal cover of $v_{0}$ or 
 $v_{0} \subset f\bparen{\pi(\talf u_{0})}$.
In the former case we are done, so suppose the latter case holds. Then $v_{0} \subset f\bparen{\pi(T\bfs)}$
and so  $\pi(T\bfs)$ has a factor $v_{1}$ such that $v_{0} \subset f(v_{1})$ and $\nabs{v_{1}} \leq \nabs{v_{0}}$.
Since $f(\calL) \subset \calL$, it follows that  $v_{1}$ is a forbidden word of $\calX$;
by taking a factor of $v_{1}$ if necessary, we may assume~$v_{1}$ is also minimal.
Let $u_{1}$ be the shortest prefix of $T\bfs$ such that $v_{1} \subset \pi(u_{1})$.
Then either $u_{1}$ is a minimal cover of $v_{1}$ or  $v_{1} \subset f\bparen{\pi(\talf u_{1})}$.
In the former case $u_{1}\in \calC$  and so $\bfs$ has a factor $u_{1}$ in $\calC$.
In the latter case  $v_{1} \subset f\bparen{\pi(T^{2}\bfs)}$, and we continue the process.
This generates a sequence $v_{0}$, $v_{1}$, \ldots\ of minimal forbidden words of $\calX$ such that
$v_{n} \subset f\bparen{\pi(T^{n+1}\bfs)}$, $v_{n + 1} \subset f(v_{n})$, and $\nabs{v_{n+1}} \leq \nabs{v_{n}}$.
Each letter $a \in A$ is growing because $f \in \calN$, and therefore the words $v_{n}$ are pairwise distinct.
Thus the length restriction on the $v_{n}$'s implies that the sequence $v_{0}, v_{1}, \dots$ is finite with a last element, say, $v_{k}$.
The fact that there is no element $v_{k+1}$ means that $T^{k+1}\bfs$ has a prefix $u_{k}$ 
that is a minimal cover of $v_{k}$. Since $u_{k}\in \calC$ then occurs also in $\bfs$, we are done.
\end{proof}

\begin{theorem}\label{2013-04-14 11:13}
Suppose that $f \in \calN$ and  that the set  $\calC$  of minimal covers of minimal forbidden words is a regular language. Then the language of $\calS$ is regular.  In particular, $(\calS, T)$ is a sofic subshift.
%
%$\bfs \in S^{\N}$ is in $\calS$ if and only if it is the label of an infinite walk in the graph of a DFA
%of the complement of $S^{*}\calC S^{*}$.
\end{theorem}
\begin{proof}
Since $\calC$ is regular, so is the complement $S^{*}\setminus S^{*}\calC S^{*}$, which we denote by $L_{0}$. 
Let $M_{0}$ be the  minimal DFA accepting $L_{0}$. Modify $M_{0}$ by removing the states from which there are no
arbitrarily long  directed walks to accepting states. Remove also the corresponding edges and denote the obtained NFA 
by~$M$.
We claim that the language $\calL(\calS)$ of $\calS$ is the language $L(M)$ recognized by $M$. 

If $w \in \calL(\calS)$, then $w$ is in $S^{*}\setminus S^{*}\calC S^{*}$ by Lemma~\ref{2013-04-11 12:48}, so that it is accepted by $M_{0}$. Furthermore, since $w$ has arbitrarily long extensions to the right that are also 
in $\calL(\calS)$, each accepted by $M_{0}$ of course, it follows that $w$ is accepted by $M$. Conversely, by the construction of $M$, if $w \in L(M)$, then there exists an infinite walk on the graph of $M$ whose label contains $w$. The label of this infinite path is in~$\calS$.
\end{proof}

\section{The suffix conjugate of the Fibonacci subshift}\label{2013-04-14 13:40}

Recall the Fibonacci morphism $\varphi$ for which $\o \mapsto \o\i$ and $\i \mapsto \o$,
the Fibonacci word $\bff = \varphi^{\omega}(\o)$, and the Fibonacci subshift $(\calX_{\varphi}, T)$.
The suffix conjugate $(\calS_{\varphi}, H_{\varphi})$ of the Fibonacci subshift is guaranteed to exist by
Theorem~\ref{2013-04-12 22:44}. The goal of this section is to give a characterization for $\calS_{\varphi}$ and $H_{\varphi}$, and it will be achieved in Theorem~\ref{2013-04-05 16:26}.

The set of suffixes of $\varphi$ is $S' = \nsset{\o, \i, \o\i}$, and we define a bijection $c$
between $S = \nsset{\o, \i, \ka}$ and $S'$ by $c(\o) = \o$, $c(\i) = \i$, and $c(\ka) = \o\i$.
In this case we have $\calS_{\varphi}  \subset \terna^{\N}$.

We will now continue by finding a characterization for the set $\calC_{\varphi}$ of minimal covers of minimal forbidden words
of the Fibonacci subshift.

Denote $f_n = \varphi^{n-1}(\o)$ for all $n\geq1$, so that in particular $f_1 = \o$ and $f_2 = \o\i$.
For $n\geq 2$, we let $p_n$ be the word defined by the relation $f_n = p_n ab$, where $ab \in \{\o\i, \i\o\}$.
Then $p_2 = \epsilon$ and  $p_3 = \o$. The words $p_{n}$ are known as the \emph{bispecial factors} of the Fibonacci word, and they possess the following well-known and easily established
 properties:
\begin{itemize}
\item For all $n \geq 2$, we have
\begin{equation}\label{2013-04-14 02:13}
f_{n} f_{n-1} = p_{n+1} ab  \qqtext{and}  f_{n-1} f_{n} = p_{n+1} ba,
\end{equation}
where $ab=\i\o$ for even $n$ and $ab=\o\i$ for odd $n$.
\item For all $n\geq 2$, we have $\varphi(p_n)\o = p_{n+1}$.
\end{itemize}

The minimal forbidden words of the Fibonacci word $\bff$ can be expressed in terms of the bispecial factors $p_{n}$ as follows~\cite{MigResSci2002}. For every $n\geq2$, write
\[
d_{n} = 
\begin{cases}
\i p_{n} \i & \text{for $n$ even,}\\
\o p_{n} \o& \text{for $n$ odd.}
\end{cases}
\] 
Then a word is a minimal forbidden word of $\bff$ if and only if it equals $d_{n}$ for some $n\geq2$.
The first few $d_{n}$'s are $\i\i$, $\o\o\o$, and $\i\o\i\o\i$.

If $x$ is a finite word and $y$ a finite or infinite word, we write ($x <_{p} y$) $x \leq_{p} y$ to indicate that $x$ is a (proper) prefix of~$y$. We say two finite words $x,y$ are \emph{prefix compatible} if one of $x \leq_{p} y$ or 
$y \leq_{p} x$ holds.

\begin{lemma}\label{2013-03-28 17:54}
Let $x,y \in \bina^{+}$ and $k \geq 1$. Then $\varphi^{k}(x) <_{p} \varphi^{k}(y)$ implies 
$x^{\flat} <_{p} y^{\flat}$.
\end{lemma}
\begin{proof}
Suppose $x^{\flat}$ is not a prefix of $y^{\flat}$. Then $x = uat$ and $y=ubs$ with distinct letters $a, b$ and 
nonempty words $t,s$. Then one of $\varphi(at)$ and $\varphi(bs)$ starts with $\o\i$ and the other one with $\o\o$.
Thus $\varphi(x)$ is not a prefix of $\varphi(y)$. The rest follows by induction.
\end{proof}

\begin{lemma} \label{2013-04-14 01:50}
Let $x,y \in \terna^*$ and suppose that $\pi(x) <_{p} \pi(y)$. Then either 
$x^{\flat} <_{p} y^{\flat}$ or $x = u \o \i$ and $y = u \ka s$ for some $u \in \terna^{*}$
and  nonempty $s \in \terna^{+}$.
\end{lemma}
\begin{proof}
Suppose $x^{\flat}$ is not a prefix of~$y^{\flat}$. 
Then $x = uat$ and $y=ubs$ with distinct letters $a, b \in \terna$ and  nonempty words $t,s \in \terna^{+}$.
Lemma~\ref{2013-04-10 17:21} applied to finite words gives
\[
\pi(x) = \pi(u) \varphi^{\nabs{u}}\bparen{\pi(at)} 
\qqtext{and}
e(y) = \pi(u) \varphi^{\nabs{u}}\bparen{\pi(bs)} 
\]
Thus $\pi(x) <_{p} \pi(y)$ implies 
$\varphi^{\nabs{u}}\bparen{\pi(at)}<_{p} \varphi^{\nabs{u}}\bparen{\pi(bs)}$,
so that by Lemma~\ref{2013-03-28 17:54}, we have $\pi(at)^{\flat} <_{p}  \pi(bs)^{\flat}$, or
\[
c(a) \varphi(\pi(t))^{\flat} <_{p} c(b) \varphi(\pi(s))^{\flat}.
\]
Since $a\neq b$, it follows that $a = \o$, $b = \ka$, and $\varphi(\pi(t))^{\flat} = \epsilon$.
The last identity implies $t = \i$; therefore $x = u \o \i$ and $y = u \ka s$. 
\end{proof}

\begin{lemma} \label{2013-03-28 15:43}
We have $d_{3} = \pi(\o\i) \o$ and $d_{4} = \pi(\i\o) \o\i$.
For all $n\geq0$, we have
\[
\pi(\o \ka \i^{2n}\ka) =  d_{2n+5}\i  \qqtext{and}
\pi(\i\ka \i^{2n+1}\ka) =  d_{2n + 6} \o
\]
\end{lemma}
\begin{proof}
Recalling that $f_{k} f_{k+1} = p_{k+2}ab$ with $ab \in \{\o\i, \i\o\}$ for all $k\geq1$, we get
\begin{align*}
\pi(\o \ka \i^{2n} \ka) &= \o \varphi(\o\i) \varphi^{2}(\i) \varphi^{3}(\i) \cdots \varphi^{2n}(\i)\varphi^{2n + 1}(\i) \varphi^{2n+2}(\o\i)  \\
&= \o \varphi^{2}(\o) \varphi^{1}(\o) \varphi^{2}(\o) \cdots \varphi^{2n-1}(\o)\varphi^{2n}(\o) \varphi^{2n+3}(\o)  \\
&= \o f_{3} f_{2} f_{3} \cdots f_{2n} f_{2n+1}  f_{2n+4}  \\
&= \o f_{2n+3} f_{2n+4} = \o p_{2n+5} \o \i = d_{2n+5} \i.
\end{align*}
Similarly,
\begin{align*}
\pi(\i \ka \i^{2n+1} \ka) &= \i \varphi(\o\i) \varphi^{2}(\i) \varphi^{3}(\i) \cdots \varphi^{2n}(\i)\varphi^{2n + 1}(\i) \varphi^{2n+2}(\i)  \varphi^{2n+3}(\o\i)  \\
&= \i \varphi^{2}(\o) \varphi^{1}(\o) \varphi^{2}(\o) \cdots  \varphi^{2n-1}(\o) \varphi^{2n}(\o)\varphi^{2n+1}(\o) \varphi^{2n+4}(\o)  \\
&= \i f_{3} f_{2} f_{3} \cdots f_{2n+1} f_{2n+2}  f_{2n+5}  \\
&= \i f_{2n+4} f_{2n+5} = \i p_{2n+6} \i \o = d_{2n+6} \o.
\end{align*}
\end{proof}

\begin{lemma}\label{2013-04-14 01:49}
The forbidden word $d_{2} = \i\i$ does not have covers.
The minimal covers of $d_{3}$ are the words in $\o\i(\o+\i+\ka)$.
The minimal covers of $d_{4}$ are the words in $(\i + \ka)\o(\o+\i+\ka)$.
%The minimal covers of $d_{5}$ are the words in $\o\ka\o(\o+\i+\ka)$ and $\o\ka\ka$.
%The minimal covers of $d_{6}$ are the words in $\i\ka\i\o(\o+\i+\ka)$ and $\i\ka\i\ka$.
%The minimal covers of $d_{7}$ are the words in $\o\ka \i\i\o(\o+\i+\ka)$,
%$\o\ka\i\i\ka$.
For other forbidden words, we have the following. Let $n\geq0$.
\begin{enumerate}[(i)]
\item The minimal covers of $d_{2n+5}$ are 
\begin{equation}\label{2013-04-14 10:44}
\o\ka\i^{2n}\bparen{\ka + \o\o + \o\i + \o\ka}.
\end{equation}
\item The minimal covers of $d_{2n+6}$ are
\[
(\i + \ka)\ka\i^{2n+1}\bparen{\ka + \o\o + \o\i + \o\ka}.
\]
\end{enumerate}
\end{lemma}
\begin{proof}
We leave verifying the claims on $d_{2}$, $d_{3}$, and $d_{4}$ to the reader.
The displayed words are minimal covers because they are obtained from the clearly minimal words in
Lemma~\ref{2013-03-28 15:43} by modifying the first and the last two letters in obvious ways.

To prove that this collection is exhaustive, suppose that $u$ is a minimal cover of $d_{2n+5}$. Then $d_{2n + 5} \subset \pi(u)$, and since $d_{2n+5}$ is not a factor of $\varphi\bparen{\pi(\talf u)}$,
it follows that $u$ can be written as $u=ax$ with $a\in \terna$ such that $d_{2n+5} \leq_{p} \pi(bx)$ for some $b \in \terna$. Noticing that $d_{2n+5}$ starts with with $\o\o$, we actually must have $a = b = \o$, and so
$d_{2n+5} \leq_{p} \pi(u)$. Lemma~\ref{2013-03-28 15:43} says that then $\pi(\o \ka \i^{2n}\ka) <_{p} \pi(u)$, so that
$\o \ka \i^{2n} <_{p} u^{\flat}$ by Lemma~\ref{2013-04-14 01:50}. It is readily verified using~\eqref{2013-04-14 02:13} that 
$\pi(\o \ka \i^{2n}\i)$ is not prefix compatible with $d_{2n + 5}$, and thus either $\o \ka \i^{2n}\o \leq_{p} u$
or $\o \ka \i^{2n}\ka \leq_{p} u$. This observation and the minimality of $u$ show that the words in~\eqref{2013-04-14 10:44} are exactly all the minimal covers of $d_{2n+5}$.

The case for $d_{2n+6}$ can be handled in the same way, the only difference being that since $d_{2n+6}$ starts with $\i\o$,
the letters $a$ and $b$ may differ, but then $\nsset{a,b} = \nsset{\i, \ka}$.
\end{proof}

\begin{theorem}\label{2013-04-05 16:26}
The language $\calL(\calS_{\varphi})$ of the suffix conjugate $(\calS_{\varphi}, H_{\varphi})$ 
of the Fibonacci subshift $(\calX_{\varphi}, T)$ is regular.
An infinite word $\bfs \in S^{\N}$ is in $\calS_{\varphi}$ if and only if it is the label of an infinite walk on
the graph depicted in Fig.~\ref{2013-04-04 14:20}.
The mapping $H_{\varphi} \colon \calS_{\varphi} \rightarrow \calS_{\varphi}$ is given by 
\[
H_{\varphi}(\bfs) =
\begin{cases}
\i \bfz & \text{if $\bfs = \ka\bfz$;}\\
\lambda(x \ka) \bfz & \text{if $\bfs = ax\ka\bfz$ with $a\in\bina$, $x\in \bina^{*}$;} \\
\lambda(\bfz) & \text{if $\bfs = a\bfz$ with $a\in \bina$ and $\bfz\in \bina^{\N}$,}
\end{cases}
\]
where $\lambda$ is the morphism given by $\lambda(\i) = \o$, $\lambda(\o) = \ka$, and $\lambda(\ka) = \ka \i$.
\end{theorem}
\begin{proof}
Lemma~\ref{2013-04-14 01:49} says that the set $\calC_{\varphi}$ of all minimal covers of minimal forbidden words of $\bff$
is regular. Thus Theorem~\ref{2013-04-14 11:13} tells us that $\calL(\calS_{\varphi})$ is regular.
Following the proof of that theorem, we first construct
the minimal  deterministic automaton%
\footnote{This automaton as well as the one in the proof Theorem~\ref{2013-04-14 18:14} was computed using Petri Salmela's FAFLA Python package for finite automata and formal languages. \url{http://coyote.dy.fi/~pesasa/fafla/}} % 
accepting the language $S^{*}\setminus S^{*}\calC_{\varphi} S^{*}$ and then remove the states and edges
that cannot be on the path of an infinite walk through accepting states. The result is given in Fig.~\ref{2013-04-05 16:25}.
Notice that the label of each walk starting from state $q_{0}$ can be obtained from a walk starting from states 
$q_{1}$, $q_{3}$, or $q_{4}$. The state $q_{2}$ is superfluous for the same reason. The removal of states $q_{0}$ and $q_{2}$ and the corresponding edges yields in the graph in Fig.~\ref{2013-04-04 14:20}.

Using Eq.~\eqref{2013-04-13 21:25} for constructing the mapping $G$ and then
recalling the definition $H = T\circ G$,  the given formula for $H_{\varphi}$ is readily verified.
This completes the proof.
\end{proof}

\begin{figure}
\begin{center}
\subfloat[%
An NFA accepting the language of $\calL(\calS_{\varphi})$.%
] % provided by subfig-package
{
\label{2013-04-05 16:25}
\begin{tikzpicture}[->, >=stealth', shorten >=1pt, auto, node distance=2.3cm,%                                             
semithick, initial/.style={initial by arrow, initial text=, initial where = above}]
\tikzstyle{every state}=[semithick]
\node[state, initial, double, inner sep=1pt, minimum size=12pt] (q0)              {$q_{0}$};
\node[state, double, inner sep=1pt, minimum size=12pt] (q1) [below left of = q0] {$q_{1}$};
\node[state, double, inner sep=1pt, minimum size=12pt] (q2) [below right of = q0] {$q_{2}$};
\node[state, double, inner sep=1pt, minimum size=12pt] (q3) [below of = q1] {$q_{3}$};
\node[state, double, inner sep=1pt, minimum size=12pt] (q4) [below of = q2] {$q_{4}$};
\path (q0) edge node [left] {$\o$}  (q1)
      (q0) edge [right] node {$\i, \ka$} (q2)
      (q1) edge [out = 210, in = 150, loop] node {$\o$} (q1)
      (q2) edge [out = 30, in = -30, loop] node {$\i$} (q2)
      (q2) edge [right] node {$\ka$} (q4)      
      (q4) edge [bend left] node {$\i$} (q3)
      (q3) edge [bend left] node {$\i$} (q4)
      (q4) edge [out = 30, in = -30, loop] node {$\ka$} (q4)
      (q1) edge [left] node {$\ka$} (q3);
\end{tikzpicture}
}
\,
\subfloat[A graph for the sequences in $\calS_{\varphi}$.]
{
\label{2013-04-04 14:20}
\begin{tikzpicture}[->, >=stealth', shorten >=1pt, auto, node distance=2.3cm,%                                             
semithick, initial/.style={initial by arrow, initial text=}]
\tikzstyle{every state}=[semithick]

\node[state, inner sep=1pt, minimum size=12pt] (q0)                 {$q_{1}$};
\node[state, inner sep=1pt, minimum size=12pt] (q1) [right of = q0] {$q_{3}$};
\node[state, inner sep=1pt, minimum size=12pt] (q2) [right of = q1] {$q_{4}$};
\path (q0) edge [bend left] node {$\ka$}  (q1)
      (q0) edge [out = 120, in = 60, loop] node {$\o$} (q0)
      (q1) edge [bend left] node {$\i$} (q2)
      (q2) edge [out = 120, in = 60, loop] node {$\ka$} (q2)
      (q2) edge [bend left] node {$\i$} (q1);
\end{tikzpicture}
}
\end{center}
\caption{The suffix conjugate of the Fibonacci subshift.}
\end{figure}
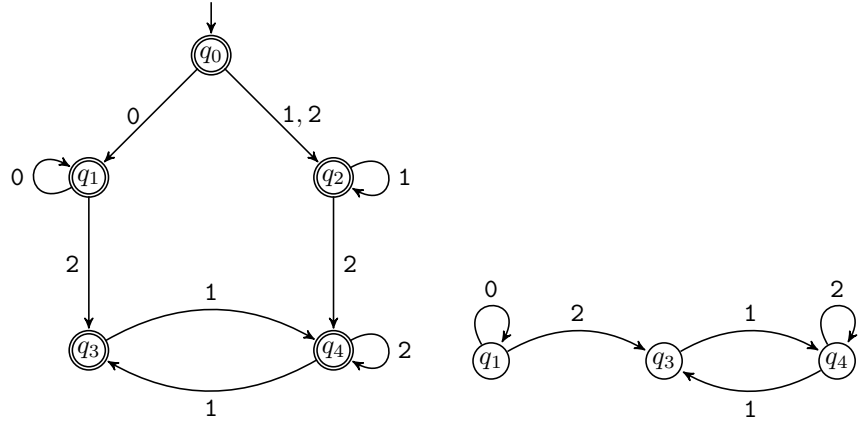

Our last goal for this section is to prove Theorem~\ref{2013-04-06 13:09}.
To that end, let us first state a well-known property of the Fibonacci subshift.

\begin{lemma}\label{2013-07-18 18:28}
If $\bfz$ has two $T$-preimages in $\calX_{\varphi}$, then $\bfz = \bff$.
\end{lemma}
\begin{proof}
If $\bfz$ has two $T$-preimages, then $\o \bfz, \i \bfz \in \calX_{\varphi}$. This means that all prefixes of $\bfz$ are so-called left special factors 
of the Fibonacci word $\bff$. The unique word in $\calX_{\varphi}$ with this property is $\bff$; see, e.g.,~\cite[Ch.~2]{Lothaire2002}.
\end{proof}

We say that an infinite word $\bfx$ is in the \emph{strictly positive orbit} of $\bfz$ if $T^{k}\bfz  = \bfx$ for some $k > 0$
and that $\bfx$ is the \emph{strictly negative orbit} of $\bfz$ if $T^{k}\bfx = \bfz$ for some $k > 0$.
Also, an infinite word $\bfx$ is said to have a \emph{tail} $\bfz$ if $\bfx = u \bfz$ for some finite word $u$.

Since $\bff = \o \i \varphi(\i) \varphi^{2}(\i) \cdots$, we have $\bff = \pi(\ka \i^{\omega})$.
Thus  Theorem~\ref{2013-04-05 16:26} gives
\[
T\bff = T \circ \pi (\ka\i^{\omega}) = \pi \circ H_{\varphi}(\ka\i^{\omega}) = \pi(\i^{\omega}).
\]
Similarly,
\[
T^{2} \bff = \pi(\o^{\omega}) \qqtext{and} T^{3} \bff = \pi(\ka^{\omega})
\]
The word $T^{2} \bff$ divides the shift orbit of the Fibonacci word in the following sense.

\begin{theorem}\label{2013-04-06 13:09}
Let $\bfs \in \calS_{\varphi}$. Then 
\begin{enumerate}[(i)]
\item $\pi(\bfs)$ is in the strictly positive orbit of $T^{2}\bff$ if and only if $\bfs$ has a tail $\ka^{\omega}$.\label{itemF1}
\item $\pi(\bfs)$ is in the strictly negative orbit of $T^{2}\bff$ if and only if $\bfs$ has a tail $\i^{\omega}$.\label{itemF2}
\end{enumerate}
\end{theorem}
\begin{proof}
We begin by proving \eqref{itemF1}.
If $\pi(\bfs)$ is in the strictly positive orbit of $T^2 \bff$, then there exists $k > 0$ such that
\[
\pi(\bfs) = T^{k + 2} \bff = T^k \circ \pi(\o^{\omega}) = \pi \circ H_{\varphi}^k(\o^{\omega}) = \pi \circ H_{\varphi}^{k - 1}(\ka^{\omega}).
\]
Since $\pi$ is injective, we have $\bfs = H_{\varphi}^{k - 1}(\ka^{\omega})$. 
Thus we see from the characterization of $H_{\varphi}$ given in Theorem~\ref{2013-04-05 16:26} that $\bfs$ has a tail $\ka^{\omega}$.

Conversely, suppose that $\bfs$ has a tail $\ka^{\omega}$. Then there exists an integer $m\geq 0$ such that 
both $\pi(\bfs)$ and $T^3\bff$ have a tail $f^m\bparen{\pi(\ka^{\omega})}$.
This implies that there exist finite words $u,v$ and $\bfz \in \{\o, \i\}^{\N}$ such that $\pi(\bfs) = u\bfz$ and $T^3\bff = v\bfz$ and the only common suffix of 
$u$ and $v$ is the empty word $\epsilon$.

\emph{Case 1:} $u \neq \epsilon$ and $v \neq \epsilon$. Then $\o \bfz, \i \bfz \in \calX_{\varphi}$,
and so by Lemma~\ref{2013-07-18 18:28}, we have $\bfz = \bff$. But $\bff$ cannot be a tail of $T^3\bff$ because $\bff$ is aperiodic, a contradiction.

\emph{Case 2:} $u \neq \epsilon$ and $v = \epsilon$. Then both $T^{\nabs{u} - 1} \circ \pi(\bfs)$ and $T^2 \bff$ are $T$-preimages of $T^3\bff$,
so that $T^{\nabs{u} - 1} \circ \pi(\bfs) = T^2\bff$ again by Lemma~\ref{2013-07-18 18:28}.
Thus 
\[
\pi(\o^{\omega}) = T^2 \bff = T^{\nabs{u} - 1} \circ \pi(\bfs) = \pi \circ H_{\varphi}^{\nabs{u} - 1}(\bfs),
\]
from which we get $H_{\varphi}^{\nabs{u} - 1}(\bfs) = \o^{\omega}$ by the injectivity of $\pi$. But this is not possible because $\bfs$ has a tail $\ka^{\omega}$ and 
Theorem~\ref{2013-04-05 16:26} says that $H_{\varphi}$ preserves such tails.

\emph{Case 3:} $u = \epsilon$. Then $\pi(\bfs)$ is in the strictly positive orbit of $T^2\bff$, and this is what we wanted to prove.

Let us then prove \eqref{itemF2}. If $\pi(\bfs)$ is in the negative orbit of $T^2 \bff$, then
$T^k \pi(\bfs) = T^2 \bff$ for some $k > 0$. Then
\[
\pi(\i^{\omega}) = 
T \bff = T^{k - 1} \circ  \pi(\bfs) = 
\pi \circ H_{\varphi}^{k - 1}(\bfs),
\]
so that $H_{\varphi}^{k - 1}(\bfs) = \i^{\omega}$ by the injectivity of $\pi$.
The characterization of $H_{\varphi}$ given in Theorem~\ref{2013-04-05 16:26} 
shows, then, that $\bfs$ must have a tail $\i^{\omega}$.

Conversely, suppose that $\bfs$ has a tail $\i^{\omega}$. Since $T\bff = \pi(\i^{\omega})$, 
there exists an integer $m\geq 0$ such that 
both $\pi(\bfs)$ and $T\bff$ have a common tail $f^m\bparen{\pi(\i^{\omega})}$.
Thus $\pi(\bfs) = u\bfz$ and $T\bff = v\bfz$ for some finite words $u,v$ and $\bfz \in \{\o, \i\}^{\N}$ such that the only common suffix of 
$u$ and $v$ is the empty word $\epsilon$. 

\emph{Case 1:} $v \neq \epsilon$ and $u \neq \epsilon$. Then $\o\bfz, \i\bfz \in \calX_{\varphi}$, so that $\bfz = \bff$, contradicting the fact that $\bff$ cannot be a tail of $T\bff$.

\emph{Case 2:} $v \neq \epsilon$ and $u = \epsilon$.  Then
\[
\pi(\bfs) = T^{\nabs{v}} \circ T\bff = T^{\nabs{v}} \circ \pi(\i^{\omega})  = \pi \circ H_{\varphi}^{\nabs{v}} (\i^{\omega})
=  \pi \circ H_{\varphi}^{\nabs{v} - 1} (\o^{\omega}),
\]
so that $ \bfs = H_{\varphi}^{\nabs{v} - 1} (\o^{\omega})$. Once again, the characterization of $H_{\varphi}$ shows that $\i^{\omega}$ cannot be a tail of $\bfs$, 
a contradiction.

\emph{Case 3:} $v = \epsilon$. Then  $\pi(\bfs)$ is in the strictly negative orbit of $T^2\bff$, which is want we wanted to show.
\end{proof}

\section{The suffix conjugate of the Thue-Morse subshift}\label{2013-04-14 13:38}

Let $\mu$ be the Thue-Morse morphism $\o \mapsto \oi$, $\i \mapsto \io$ and $\bft = \mu^{\omega}(\o)$
the Thue-Morse word.  Let $\calX_{\mu}$ denote the shift orbit closure of $\bft$, so that 
the Thue-Morse subshift is $(\calX_{\mu}, T)$.
In this section we will characterize its suffix conjugate $(\calS_{\mu}, H_{\mu})$ defined in Theorem~\ref{2013-04-12 22:44}.

Here the set of suffixes is $S' = \nsset{\o, \i, \o\i, \i\o}$ and $S = \nsset{\o, \i, \ka, \ko}$,
and we let~$c$ be the bijection between $S$ and $S'$ given by
\[
c(\o) = \o, \qquad c(\i) = \i, \qquad c(\ka) = \o\i, \qquad c(\ko) = \i\o.
\]

The minimal forbidden words of the Thue-Morse word
are $\o\o\o$, $\i\i\i$,
\begin{align*}
\o\mu^{2n}(\o\i\o) \o,& \qquad \o\mu^{2n}(\i\o\i) \o, \qquad \i\mu^{2n}(\o\i\o) \i, 
\qquad \i\mu^{2n}(\i\o\i)\i
\\
\intertext{and}
\i\mu^{2n+1}(\o\i\o) \o,& \qquad \i\mu^{2n+1}(\i\o\i) \o, \qquad \o\mu^{2n+1}(\o\i\o) \i,
\qquad \o\mu^{2n+1}(\i\o\i) \i 
\end{align*}
for all $n\geq 0$; see ~\cite{MigResSci2002,Shur2005}. 

Let us introduce a shorthand.
For $x,y,z \in \bina$ and $k\geq0$, we write
\[
\gamma(k,x,y,z) = x \mu^{k}(y \overline{y} y) z.
\]
Here the overline notation $\overline{\cdot}$  swaps $\o$'s and $\i$'s.
The minimal forbidden words of $\bft$ can then be written as $xxx$ and 
\begin{equation}\label{2013-04-09 22:16}
\gamma(2n, x, x, x), \quad \gamma(2n, x, \overline{x}, x),
\quad \gamma(2n + 1, x, x, \overline{x}), \quad \gamma(2n + 1, x, \overline{x}, \overline{x}),
\end{equation}
for all $n\geq 0$ and $x \in \bina$. Furthermore, 
\begin{itemize}
\item $\mu\bparen{\gamma(k, x, y, z)} = x \gamma(k+1, \overline{x}, y, z) \overline{z}$, and
\item  $\gamma(k, x, y, z)$ is a forbidden word if and only if 
$\gamma(k - 1, \overline{x}, y, z)$ is a forbidden word, where $k\geq1$.
\end{itemize}

The mapping $\lambda$ defined in~\eqref{2013-04-13 20:46} and~\eqref{2013-04-14 16:13} in the current case 
is 
\[
\lambda(\o) = \ka, \qquad \lambda(\i) = \ko, \qquad \lambda(\ka) = \ka \i , \qquad \lambda(\ko) = \ko\o.
\]

The next lemma is analogous to Lemma~\ref{2013-04-14 01:50} for the Fibonacci subshift.

\begin{lemma} \label{2013-03-28 19:40}
Let $x,y \in \kvarta^*$ and suppose that $\pi(x) <_{p} \pi(y)$.
Then one of the following holds:
\begin{enumerate}[(i)]
\item $x^{\flat} <_{p} y^{\flat}$ 
\item $x = u z \overline{z}$ and $y = u \lambda(z) s$, where $z \in \bina$ and $s$ has prefix $z$ or $\lambda(z)$.
\item $x = u \lambda(z)z$ and $y = uzs$, where $z \in \bina$ and $s$ has prefix $\overline{z}\overline{z}$ or $\overline{z}\lambda(\overline{z})$.
\end{enumerate}
\end{lemma}
\begin{proof}
Suppose that $x^{\flat}$ is not a prefix of~$y^{\flat}$. 
Then $x = uat$ and $y=ubs$ with distinct letters $a, b \in \kvarta$ 
and  nonempty words $t,s \in \kvarta^{+}$.
Lemma~\ref{2013-04-10 17:21} gives
\[
\pi(x) = \pi(u) \mu^{\nabs{u}}\bparen{\pi(at)} 
\qqtext{and}
\pi(y) = \pi(u) \mu^{\nabs{u}}\bparen{\pi(bs)},
\]
so that $\pi(x) <_{p} \pi(y)$ implies 
$\mu^{\nabs{u}}\bparen{\pi(at)} <_{p} \mu^{\nabs{u}}\bparen{\pi(bs)}$.
Since $\mu$ is an injective and $\nabs{\mu(\o)} = \nabs{\mu(\i)}$, it follows that
$\pi(at) <_{p}  \pi(bs)$, or
\[
c(a) \mu\bparen{\pi(t)} <_{p} c(b) \mu\bparen{\pi(s)}.
\]

Since $a \neq b$, this implies that $\nsset{a,b} = \nsset{z, \lambda(z)}$ for some $z \in \bina$.
If $a = z$ and $b = \lambda(z)$, a simple analysis shows that we have $t = \overline{z}$ and $s$ starts with $z$
or with $\lambda(z)$;
this corresponds to option~(ii).
Similarly, if $a = \lambda(z)$ and $b = z$, then $t=z$ and $s$ starts with $\overline{z}\overline{z}$ or $\overline{z}\lambda(\overline{z})$; this corresponds to option~(iii).
\end{proof}

Our next goal is to characterize the minimal covers of the minimal forbidden words of $\calX_{\mu}$.
We begin with the next lemma, whose easy verification is left to the reader.

\begin{lemma}\label{2013-04-09 22:13}
Let $x\in \bina$.
The forbidden words $xxx$ and $\gamma(0, x, x, x)$ do not have covers.
For other forbidden words, we have the following.
\begin{enumerate}[(i)]
\item The minimal covers of $\gamma(0, x, \overline{x}, x)$ are in
\[
\bparen{x + \lambda(\overline{x})}\overline{x}\bparen{\overline{x} + \lambda(\overline{x})}
\qqtext{and}
\lambda(x) x \bparen{x + \lambda(x)}.
\]
\item The minimal covers of $\gamma(1, x, x, \overline{x})$ are in
\[
\bparen{x + \lambda(\overline{x})} x \overline{x} \bparen{\overline{x} + \lambda(\overline{x})}
\qqtext{and}
\bparen{x + \lambda(\overline{x})} \lambda(x) \bparen{x + \lambda(x)}.
\]
\item The minimal covers of $\gamma(1, x, \overline{x}, \overline{x})$ are in
\[
\bparen{x + \lambda(\overline{x})} \overline{x} x \bparen{\overline{x} + \lambda(\overline{x})}
\qqtext{and}
\bparen{x + \lambda(\overline{x})} \overline{x} \lambda(x).
\]
\end{enumerate}
\end{lemma}

We will characterize the minimal covers of the remaining forbidden words in Lemma~\ref{2013-04-09 16:05}
with the help of the following result.

\begin{lemma}\label{2013-04-08 16:49}
For $x,y,z \in \bina$ and $k\geq 2$, we have
\begin{align}
\gamma(k, x, y, z) = \pi\bparen{x\lambda(y) \overline{y}^{k-2}\lambda(\overline{y})} z. \label{2013-04-08 15:43}\\
\intertext{Furthermore, $x\lambda(y) \overline{y}^{k-1} w$ is a minimal cover of $\gamma(k, x, y, \overline{y})$ with}
\gamma(k, x, y, \overline{y}) <_{p} \pi\bparen{x\lambda(y) \overline{y}^{k-1} w} \label{2013-04-14 23:17}
\end{align}
if and only if $w \in \nsset{y, \lambda(y)}$. 
\end{lemma}
\begin{proof}
Observe that
\begin{align}
\mu^{k}(y) &= y\overline{y} \mu(\overline{y}) \mu^{2}(\overline{y}) \cdots \mu^{k-1}(\overline{y})\nonumber\\
&= y \overline{y} \overline{y} y \mu^{2}(\overline{y}) \cdots \mu^{k-1}(\overline{y})\nonumber \\
&=  \mu\bparen{c(\lambda(y))}  \mu^{2}(\overline{y}) \cdots \mu^{k-1}(\overline{y}) \nonumber\\
\intertext{so that}
x \mu^{k}(y) 
&= x \mu\bparen{c(\lambda(y))} \mu^{2}(\overline{y}) 
\cdots \mu^{k-1}(\overline{y}) = \pi\bparen{x \lambda(y) \overline{y}^{k-2}} \label{2013-04-15 00:00}.
\end{align}
Consequently, we have
\[
x \mu^{k}(y\overline{y} y) z = 
x \mu^{k}(y) \mu^{k}\bparen{c(\lambda(\overline{y}))} z = 
\pi\bparen{x \lambda(y) \overline{y}^{k-2} \lambda(\overline{y})} z,
\]
verifying~\eqref{2013-04-08 15:43}. We also have
\begin{align*}
x \mu^{k}(y\overline{y} y) &= x \mu^{k}(y) \mu^{k}(\overline{y}) \mu^{k}(y)
= \pi\bparen{x \lambda(y) \overline{y}^{k-1}} \mu^{k}(y),
\end{align*}
where the latter identity is due to~\eqref{2013-04-15 00:00} and Lemma~\ref{2013-04-10 17:21}.
From this it follows that
\[
\gamma(k, x, y, \overline{y}) = x \mu^{k}(y\overline{y} y)\overline{y}  <_{p}  
\pi\bparen{x \lambda(y) \overline{y}^{k-1} w}
\]
if and only if $w \in \nsset{y, \lambda(y)}$,
in which case $x \lambda(y) \overline{y}^{k-1} w$ is a minimal cover of $\gamma(k, x, y, \overline{y})$.
\end{proof}

\begin{lemma}\label{2013-04-09 16:05}
Let $x,y \in \bina$ and $k\geq 2$.
A word is a minimal cover of $\gamma(k, x, y, y)$
if and only if it is in
\begin{equation}\label{2013-04-14 21:45}
\bparen{x + \lambda(\overline{x})}\lambda(y) \overline{y}^{k-2} \lambda(\overline{y}) \bparen{y + \lambda(y)}.
\end{equation}
A word is a minimal cover of $\gamma(k, x, y, \overline{y})$
if and only if it is in
\begin{equation}\label{2013-04-14 23:19}
\bparen{x + \lambda(\overline{x})}\lambda(y) \overline{y}^{k-2}%
\left[\lambda(\overline{y}) \bparen{\overline{y} + \lambda(\overline{y})} + \overline{y}\bparen{y + \lambda(y)}\right].
\end{equation}
\end{lemma}
\begin{proof}
We see from~\eqref{2013-04-08 15:43} that the words in~\eqref{2013-04-14 21:45} really are minimal covers of
$\gamma(k, x, y, y)$. (Notice here that $c(x)$ is a suffix of $c(\lambda(\overline{x}))$
and $c(y)$ is a prefix of $c(\lambda(y))$.) Further, we see from~\eqref{2013-04-08 15:43} 
and~\eqref{2013-04-14 23:17} that the words in~\eqref{2013-04-14 23:19} 
really are minimal covers of $\gamma(k, x, y, \overline{y})$.

Let us show the converse. Let $u \in \kvarta^{*}$ be a minimal cover of $\gamma(k,x,y,z)$ with $z \in \nsset{y, \overline{y}}$.
Then there exists $a,b \in \nsset{x, \lambda(\overline{x})}$ and $t\in \kvarta^{*}$ such that $u = at$ and $\gamma(k,x,y,z) \leq_{p} \pi(bt)$. Since both possibilities for $a$ are accounted for in~\eqref{2013-04-14 21:45}
and~\eqref{2013-04-14 23:19},
we may assume that $a = b$, and so $u = bt$. Then \eqref{2013-04-08 15:43} gives
\begin{equation}\label{2013-04-08 22:52}
\pi\bparen{x\lambda(y) \overline{y}^{k-2}\lambda(\overline{y})} =
\gamma(k,x,y,z) z^{-1} <_{p} \pi(u).
\end{equation}
Now Lemma~\ref{2013-03-28 19:40} applies, and since its options (ii) and (iii) are clearly not possible here, we get
\[
x\lambda(y) \overline{y}^{k-2} <_{p} u^{\flat},
\]
and so we have $u = x\lambda(y) \overline{y}^{k-2} w$ for some $w \in \kvarta^{+}$.
The minimality of~$u$ implies $\nabs{w} = 2$.
Equation~\eqref{2013-04-08 22:52} and Lemma~\ref{2013-04-10 17:21} then imply
\[
\mu^{k}\bparen{\pi(\lambda(\overline{y}))} z <_{p} \mu^{k}(\pi(w)),
\qqtext{or equivalently,}
\pi(\lambda(\overline{y}))  z <_{p} \pi(w).
\]
If $z = y$, then $w \in \lambda(\overline{y})\bparen{y + \lambda(y)}$.
If $z = \overline{y}$, then either 
\[
w \in \overline{y}\nparen{y + \lambda(y)} \qqtext{or}
w \in \lambda(\overline{y})\nparen{\overline{y} + \lambda(\overline{y})},
\]
and this completes the proof.
\end{proof}

\begin{theorem}\label{2013-04-14 18:14}
The language $\calL(\calS_{\mu})$ of the suffix conjugate $(\calS_{\mu}, H_{\mu})$ of the Thue-Morse subshift $(\calX_{\mu}, T)$ is regular.
An infinite word $\bfs \in S^{\N}$ is in $\calS_{\mu}$ if and only if it is the label of an infinite walk on
the graph depicted in Fig.~\ref{2013-04-08 20:11}.
The mapping $H_{\mu} \colon \calS_{\mu} \rightarrow \calS_{\mu}$ is given by 
\[
H_{\mu}(\bfs) =
\begin{cases}
\i \bfz & \text{if $\bfs = \ka\bfz$;}\\
\o \bfz & \text{if $\bfs = \ko\bfz$;}\\
\lambda(\bfz) & \text{if $\bfs = a\bfz$ with $a\in \bina$ and $\bfz\in \bina^{\N}$;}\\
\lambda(x\ka)\bfz & \text{if $\bfs = ax\ka\bfz$ with $a\in\bina$, $x\in \bina^{*}$;}\\
\lambda(x\ko)\bfz & \text{if $\bfs = ax\ko\bfz$ with $a\in\bina$, $x\in \bina^{*}$,}
\end{cases}
\]
where  $\lambda$ is the morphism given by $\lambda(\o) = \ka$,  $\lambda(\i) = \ko$, $\lambda(\ka) = \ka \i$,
and $\lambda(\ko) = \ko \o$.

\end{theorem}
\begin{proof}
The set $\calC_{\mu}$ is obtained from~\eqref{2013-04-09 22:16} using Lemmas~\ref{2013-04-09 22:13} and~\ref{2013-04-09 16:05}. From these it is clear that $\calC_{\mu}$ is a regular language.
Thus $\calL(\calS_{\mu})$ is regular by Theorem~\ref{2013-04-14 11:13}.
Taking the steps outlined in the proof of that theorem and removing the superfluous states and edges,
as in the proof of Theorem~\ref{2013-04-05 16:26}, 
we get the graph depicted in Fig.~\ref{2013-04-08 20:11}.
Finally, the values of $H_{\mu}$ are obtained directly from the definition $H = T\circ G$ and~\eqref{2013-04-13 21:25}.
\end{proof}

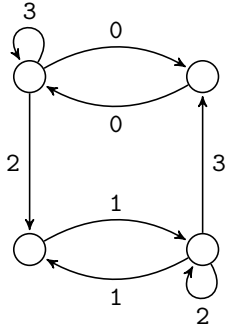
\begin{figure}
\begin{center}
\begin{tikzpicture}[->, >=stealth', shorten >=1pt, auto, node distance=2.3cm,%                                             
semithick, initial/.style={initial by arrow, initial text=}]
\tikzstyle{every state}=[semithick]
%\node[state, inner sep=1pt, minimum size=12pt] (q1) {}; %{$q_{1}$};
\node[state, inner sep=1pt, minimum size=12pt] (q3) [right of = q1] {}; % {$q_{3}$};
\node[state, inner sep=1pt, minimum size=12pt] (q5) [below of = q3] {}; %{$q_{5}$};
\node[state, inner sep=1pt, minimum size=12pt] (q4) [right of = q3] {}; %{$q_{4}$};
\node[state, inner sep=1pt, minimum size=12pt] (q6) [right of = q5] {}; %{$q_{6}$};
%\node[state, inner sep=1pt, minimum size=12pt] (q2) [right of = q6] {}; %{$q_{2}$};
\path 
%      (q1) edge [out=150, in=210, loop] node[left] {$\o$} (q1)
%      (q1) edge [above] node {$\ko$} (q3)
      (q3) edge [bend left, above] node {$\o$} (q4)
	  (q3) edge [out = 60, in = 120, loop] node[above] {$\ko$} (q3)      
      (q4) edge [bend left, below] node {$\o$} (q3)
      (q5) edge [bend left, above] node {$\i$} (q6)
      (q6) edge [bend left, below] node {$\i$} (q5)
	  (q3) edge [left] node {$\ka$} (q5)
	  (q6) edge [right] node {$\ko$} (q4)
%	  (q2) edge [below] node {$\ka$} (q6)
	  (q6) edge [out = -60, in = -120, loop] node[below] {$\ka$} (q6);
%      (q2) edge [out = 30, in = -30, loop] node[right] {$\i$} (q2);
\end{tikzpicture} 
\end{center}
\caption{The suffix conjugate of the Thue-Morse subshift.}
\label{2013-04-08 20:11}
\end{figure}

Using the characterization of $H_{\mu}$ from Theorem~\ref{2013-04-14 18:14}, it is readily verified that
\[
\bft = \mu^{\omega}(\o)  = \pi(\ka \i^{\omega}), \qquad T \bft  = \pi(\i^{\omega}),  \qquad T^2 \bft = \pi(\ko^{\omega}).  
\]

Our last goal for this section is to establish Theorem~\ref{2013-07-17 16:03},
for which we need the lemma.

\begin{lemma}\label{2013-07-17 17:13}
If $\bfx$ has two $T$-preimages in $\calX_{\mu}$, then either $\bfx = \bft$ or $\bfx = \overline{\bft}$.
\end{lemma}
\begin{proof}
Since $\bfx$ is aperiodic, it must have infinitely many prefixes $p$ such that both $p\o$ and $p\i$ occur in $\bfx$.
Since $\o p$ and $\i p$ occur in $\bfx$ as well,  the words $p$ are so-called bispecial factors of $\bft$,
whose form is known~\cite[Prop.~4.10.5]{CasNic2010}. They are 
\[
\epsilon, \quad \o, \quad \i, \quad  \mu^m(\o\i), \quad \mu^m(\i\o), \quad\mu^m(\o\i\o), \quad \mu^m(\i\o\i)
\]
for $m\geq 0$.  Therefore either $\bfx = \mu^{\omega}(\o) =\bft$ or $\bfx = \mu^{\omega}(\i) = \overline{\bft}$.
\end{proof}

We say that an infinite word $\bfx$ is in the \emph{positive orbit} of $\bfz$ if $T^{k}\bfz  = \bfx$ for some $k \geq 0$
and that $\bfx$ is the \emph{negative orbit} of $\bfz$ if $T^{k}\bfx = \bfz$ for some $k \leq 0$.
Notice that in the previous section we used the notions strictly positive and strictly negative orbits.
The next result is analogous to Theorem~\ref{2013-04-06 13:09}.

\begin{theorem}\label{2013-07-17 16:03}
Let $\bfs \in  \calS_{\mu}$. Then 
\begin{enumerate}[(i)]
\item $\pi(\bfs)$ is in the positive orbit of $T^2 \bft$ if and only if $\ko^{\omega}$ is a tail of $\bfs$. \label{itemT1}
\item $\pi(\bfs)$ is in the negative orbit of $T \bft$ if and only if $\i^{\omega}$ is a tail of $\bfs$. \label{itemT2}
\end{enumerate}
\end{theorem}
\begin{proof}
This proof is nearly identical to the proof of Theorem~\ref{2013-04-06 13:09};
therefore  we will only prove \eqref{itemT1} and leave the proof of \eqref{itemT2} to the reader.
If $\pi(\bfs)$ is in the positive orbit of $T^2 \bft  = \pi(\ko^{\omega})$, then there exists $k\geq 0$ such that
\[
\pi(\bfs) = T^{k + 2} \bft  = T^k \circ \pi(\ko^{\omega}) = \pi \circ H_{\mu}^k(\ko^{\omega}),
\]
so that $\bfs = H_{\mu}^k(\ko^{\omega})$ by the injectivity of $\pi$. Now the characterization of $H_{\mu}$ given in Theorem~\ref{2013-04-14 18:14} shows 
that $\bfs$ must have tail $\ko^{\omega}$.

Conversely, suppose that $\ko^{\omega}$ is a tail of $\bfs$.   Since $T^2 \bft  = \pi(\ko^{\omega})$, this implies that $\pi(\bfs)$ and $T^2 \bft$ have a common tail
 of the form $\mu^m\bparen{\pi(\ko^{\omega})}$ for some $m\geq 0$.
Therefore there exist finite words $u,v$ and $\bfz \in \{\o,\i\}^{\N}$ such that
 $\pi(\bfs) = u \bfz$ and $T^2 \bft  = v\bfz$ and the longest common suffix of $u$ and $v$ is the empty word $\epsilon$.

\emph{Case 1:} $u \neq \epsilon$ and $v \neq \epsilon$. Then  $\o \bfz, \i \bfz \in \calX_{\mu}$, so that $\bfz = \bft$ or $\bfz = \overline{\bft}$ by 
Lemma~\ref{2013-07-17 17:13}. But this is a contradiction because neither $\bft$ nor $\overline{\bft}$ can be a tail of $T^2\bft$.

\emph{Case 2: } $u \neq \epsilon$ and $v = \epsilon$. 
Now both $T^{\nabs{u} - 1} \circ \pi(\bfs)$ and $T \bft $ are $T$-preimages of $T^2  \bft$,
so that $T^{\nabs{u} - 1} \circ \pi(\bfs) = T \bft  $ by Lemma~\ref{2013-07-17 17:13}.
Furthermore,
\[
\pi(\i^{\omega}) = T \bft  = T^{\nabs{u} - 1} \circ \pi(\bfs) = \pi \circ H_{\mu}^{\nabs{u} - 1} (\bfs),
\]
so the the injectivity of $\pi$ implies $\i^{\omega} = H_{\mu}^{\nabs{u} - 1} (\bfs)$. But this is impossible because $\bfs$ has a tail $\ko^{\omega}$, which is
preserved by $H_{\mu}$  according to Theorem~\ref{2013-04-14 18:14}.

\emph{Case 3:} $u = \epsilon$. Then $\pi(\bfs)$ is a tail of $T^2 \bft $, and this is what we wanted to prove.
\end{proof}

%\section{Future work}
%
%Let $f \in \calN$ with an iterative fixed point that generates a subshift $(\calX, T)$.
%\begin{itemize}
%\item Prove that the language $\calL(\calS)$ is regular; i.e. $(\calS, T)$ is a sofic subshift.
%\item Prove the analogous version of Theorem~\ref{2013-04-06 13:09} for the Thue-Morse subshift.
%\item Formulate and prove an analogous version of Theorem~\ref{2013-04-06 13:09} for the general subshift $(\calX, T)$.
%\item Prove or disprove that a word $\pi(\bfs)$ with $\bfs \in \calS$ is morphic (i.e. the image of an iterative fixed point of a morphism under a letter-to-letter coding) if and only if $\bfs$ is ultimately periodic. Holton and Zamboni~\cite{HolZam2001}
%succeeded in proving an analogous result in their encoding scheme.
%\end{itemize}

\end{document}